\documentclass[journal]{IEEEtran}

\IEEEoverridecommandlockouts                              

\usepackage{xcolor}
\usepackage{graphicx,import} 
\usepackage{epsfig} 
\usepackage{mathptmx} 
\usepackage{times} 
\usepackage{amsmath} 
\usepackage{amssymb}  
\usepackage{bbm}	
\usepackage{psfrag}
\usepackage{arydshln}
\usepackage{multirow}

\usepackage{mathdots}

\usepackage{hyperref}

\newtheorem{theorem}{\textbf{Theorem}}
\newtheorem{remark}{\textbf{Remark}}
\newtheorem{lemma}{\textbf{Lemma}}
\newtheorem{proposition}{\textbf{Proposition}}

\newtheorem{definition}{\textbf{Definition}}
\newtheorem{assumption}{\bf Assumption}

\catcode`\@=11
\def\downparenfill{$\m@th\braceld\leaders\vrule\hfill\bracerd$}
\def\overparen#1{\mathop{\vbox{\ialign{##\crcr\crcr
\noalign{\kern0.4ex}
\downparenfill\crcr\noalign{\kern0.4ex\nointerlineskip}
$\hfil\displaystyle{#1}\hfil$\crcr}}}\limits}
\catcode`\@=12

\def\NN{{\mathbb N}}    
\def\RR{{\mathbb R}}    
\renewcommand{\Re}{\mathbb{R}} 
\def\diag{\texttt{Diag}}
\def\Un{\mathfrak{1}}     
\def\KR{\mathcal{K}}    
\def \HH{\mathfrak {H}} 
\def \thetaa{\theta_a}
\def \thetab{\theta_b}

\def \mat{\mathfrak{M}}

\def\CL{\texttt{Cl}}
\def \CR{\mathcal C}
\def \bu{\bar u} 			
\def\dlambda {\tilde \lambda}
\def \a1{\theta_{a1}}
\def \b1{\theta_{b1}}
\def \aa2{\theta_{a2}}
\def \bb2{\theta_{b2}}

\DeclareMathOperator*{\im}{im}

\renewcommand{\paragraph}[1]{\smallskip\noindent\textbf{#1.} }
\def\startmodif{\color{black}}
\def\stopmodif{\color{black}}

\newif\iflong
\longtrue 

\iflong 
	\newcommand{\longue}[1]{{#1}} 
	\newcommand{\courte}[1]{{}} 
\else 
	\newcommand{\longue}[1]{{}} 
	\newcommand{\courte}[1]{{#1}} 
\fi

\newif\ifchanges
\changestrue 

\ifchanges 
	\newcommand{\changement}[1]{\textcolor{black}{#1}} 
\else 
	\newcommand{\changement}[1]{{#1}} 
\fi
\iflong 
\title{State and parameter estimation: a nonlinear Luenberger observer approach (long version)}
\else 
\title{State and parameter estimation: a nonlinear Luenberger observer approach}
\fi

\author{Chouaib Afri$^{1*}$, Vincent Andrieu$^{1,3}$, Laurent Bako$^{2}$ and Pascal Dufour$^{1}$
\thanks{$^{1}$C. Afri, V. Andrieu and P. Dufour  are with the Universit\'e de Lyon, F-69622, Lyon, France
-- Universit\'e Lyon 1, Villeurbanne, France -- CNRS, UMR 5007, LAGEP, France.
 {\tt \footnotesize  afri@lagep.univ-lyon1.fr}, {\tt \footnotesize  vincent.andrieu@gmail.com}, {\tt \footnotesize  dufour@lagep.univ-lyon1.fr}
 }
\thanks{$^{3}$V. Andrieu is also with Wuppertal University,
Arbeitsgruppe Funktionalanalysis,
Gau\ss stra\ss e 20, 42097, Wuppertal, Germany.}%
\thanks{$^{2}$L. Bako is with Laboratoire Amp\`ere -- Ecole Centrale de Lyon -- Universit\'e
de Lyon, 69130 France. {\tt \footnotesize laurent.bako@ec-lyon.fr }}%
\thanks{$^{*}$Authors acknowledge the french ministry for higher education and research for funding this PhD thesis.}%
}

\begin{document}

\maketitle

\begin{abstract}
\changement{The design of a nonlinear Luenberger observer for an extended nonlinear system resulting from a parameterized linear SISO (single-input single-output) system is studied.} 
From an observability assumption of the system, the existence of such an observer is concluded.
In a second step, a 
novel algorithm for the identification of such a system is suggested. 
\startmodif
Compared to the adaptive observers available in the literature, it has the advantage to be of low dimension and to admit a strict Lyapunov function.
\stopmodif
\end{abstract}

\section{Introduction}
\changement{In this paper, the strategy of Luenberger nonlinear observer is adopted to suggest a solution to the state and parameter estimation for linear systems. 
} 

\changement{This topic has been widely studied in the literature and it is usually referred to
 as \textit{adaptive observer designs} (see the books \cite{IoannouSun_Book_2012robust,MarinoTomei_Book_95,NarendraAnnaswamy_Book_89}). 
Adaptive observer can be traced back to G. Kreisselmeier in \cite{Kreisselmeier_TAC_77}. This work has then been extended in many directions to allow time varying matrices and multi-input multi-output systems (see for instance \cite{BastinGevers_TAC_88,MarinoTomei_TAC_95,Q_Zhang_2002}). 
Most of these results are based 
on weak Lyapunov analysis in combination with LaSalle invariance principle or adaptive scheme which ensures boundedness of all signals and  asymptotic convergence of the state estimates toward the real state.
}

\changement{The nonlinear Luenberger methodology inspired from the linear case \cite{Luenberger_IEEE_TME_64} and studied in (\cite{Shoshitaishvili_TSA_90,KazantzisKravaris_SCL_98,KreisselmeierEngel_TAC_03, AndrieuPraly_CDC_2004, Andrieu_SICON_2014}) is a method which permits to design an observer based on weak observability assumptions. A particularly interesting feature of this observer is that its convergence rate can be made as large as requested (see \cite{Andrieu_SICON_2014}).}
%

\changement{
Employing the Luenberger methodology, we  introduce in this paper a novel adaptive observer.
It has the advantage to allow  a prescribed convergence rate. 
Moreover, 
its dimension is only $4n-1$ where $n$ is the order of the system. To the best of our knowledge this is lower than existing algorithms.
Moreover, in contrast to all other available approaches, a strict Lyapunov function is obtained. This allows to give an estimate of the asymptotic estimation error knowing some bounds on the disturbances.
}

\changement{
Compared to the preliminary version of this work which has been presented in \cite{AfriAndrieuBakoDufour_ACC_15},
a study is given which shows how inputs have to be generated in order to ensure convergence of the proposed algorithm. 
%
Finally, this paper can be seen as an extension of the result of \cite{PralyEtAl_MTNS_06_ObsOscillator} in which a nonlinear Luenberger observer is constructed for a harmonic oscillator which fits in the class of the studied systems.}

\changement{The paper is divided in two parts. In a first part, some general statements are given concerning the crucial steps allowing to design a nonlinear Luenberger observer for a linear system with unknown parameters.
More precisely, in Section \ref{Sec_Existence}, the existence of  a mapping $T$ is discussed.
Section \ref{Sec_Injectivity} is devoted to the study of the injectivity of the mapping $T$ assuming some observability properties. An observer is then given in Section \ref{Sec_GenConstrObs} and its robustness is studied.}

\changement{In the second part of the paper, this general framework is then adapted to the particular case of system identification problems.
In Section \ref{Sec_GenInput} a novel notion of differentially exciting system is introduced and compared with existing notions.
This notion allows to describe precisely the kind of input that allows to estimate the parameters and the state.}
\changement{In Section \ref{Sec_AppliIdentif}, a left inverse of the mapping $T$ is constructed to get the observer when considering a specific canonical  structure for the matrices $A$, $B$ and $C$. This leads to a novel solution for the identification of linear time invariant systems.}

\iflong
This paper is the long version of \cite{Afrietal_TAC_2017}.
\fi
\changement{
\longue{To simplify the presentation, most of the proofs are given in the appendix.}
\courte{A long version of the paper which contains the missing proofs and some other comments can be found in \cite{afri:hal-01232747}.}}

{\small
\noindent\textbf{Notations:}
\begin{itemize}
\item Given a matrix $A$ in $\Re^{n\times n}$, $\sigma\{A\}$ denotes its spectrum and $\sigma_{\min}\{A\}$ the eigenvalue with smallest real part. 
\item $\Un_n$ denotes the $n$ dimensional real vector composed of $1$.
\item $I_n$ denotes the $n$ dimensional identity matrix.
\item Given a $C^{j}$ function $u$: $\bar u^{(j)}(t)=\begin{bmatrix}u(t)&\dots&u^{(j)}(t)\end{bmatrix}^\top$.
\item For a vector or a matrix $|\cdot|$ denotes the usual $2$-norm.
\item Given a set $C$, $\texttt{Cl}(C)$  is its closure.
\end{itemize}}

\section{Existence of a nonlinear Luenberger observer for state and parameters estimation}
\changement{\subsection{Problem statement}\label{sec_problem}
A parameterized linear system \longue{described by the following equations} is considered:
\begin{equation}\label{Sys}
\dot x = A(\theta)x + B(\theta)u\ ,\ y=C(\theta)x,
\end{equation}
where $\theta
$ in $\Theta
 \subset\Re^{q}$ is a vector of unknown constant parameters
and $\Theta$ is a known set, $u$ in $\Re$ is a control input.
The state vector $x$ is  in $\Re^{n}$ and $y$ is the measured output in $\Re$.
Mappings $A:\Theta\rightarrow\Re^{n\times n}$, $B:\Theta\rightarrow\Re^{n\times 1}$ and $C:\Theta\rightarrow\Re^{1\times n}$  are known $C^1$ matrix valued functions.}

\changement{In the following, an asymptotic observer for the extended (nonlinear) $n+q$ dimensional system
\begin{equation}\label{eq_ExSyst}
\dot x = A(\theta)x + B(\theta)u\ , \ \dot \theta=0\ ,\ y=C(\theta)x\ 
\end{equation}
has to be designed. Following the approach developed by Luenberger for linear systems in \cite{Luenberger_IEEE_TME_64} which has been extended to nonlinear system in \cite{Shoshitaishvili_TSA_90,KazantzisKravaris_SCL_98,AndrieuPraly_CDC_2004,Andrieu_SICON_2014}, the first step is to design a $C^1$ function $(x,\theta,w)\mapsto T(x,\theta,w)$ such that
the following equation is satisfied:
\begin{equation}\label{EDP}
\begin{array}{lr}
\dfrac{\partial T}{\partial x}(x,\theta,w)\!\!&\!\![A(\theta)x +B(\theta)u] +
\dfrac{\partial T}{\partial w}(x,\theta,w)g(w,u)\\
&\\
&= \Lambda T(x,\theta,w) + L C(\theta)x \,
\end{array}
\end{equation}
where $\Lambda$ is a Hurwitz squared matrix, $L$ a column vector and $g$ is a controlled vector field which is a  degree of freedom  added to take into account the control input.
The dimensions of the matrices and of the vector field $g$ must be chosen consistently.
This will be precisely defined in the sequel. 
The interest in this mapping is highlighted if  $(z(\cdot), w(\cdot))$, the solution of the dynamical system initiated from $(z_0,w_0)$, is considered:
$$
\dot z = \Lambda z + Ly\ , \ \dot w = g(w,u)\ .
$$
Indeed, assuming completeness (of the $w$ part of the solution), for all positive time $t$:
$$
\dot {\overparen{z(t)-T(x(t),\theta,w(t))}} = \Lambda (z(t)-T(x(t),\theta,w(t)))\ .
$$
Hence, due to the fact that $\Lambda$ is Hurwitz, asymptotically it yields
\begin{equation}\label{eq:Cvge-in-imT}
\lim_{t\rightarrow +\infty} |z(t)-T(x(t),\theta,w(t))| = 0\ .
\end{equation}
In other words, $z$ provides an estimate of the function $T$.}

\changement{The second step of the Luenberger design is to left invert the function $T$ in order to reconstruct the extended state $(x,\theta)$ 
from the estimate of $T$.
Hence,  a mapping $T^*$ has to be constructed such that 
\begin{equation}\label{eq:Inverse-Mapping-T}
T^*(T(x,\theta,w),w)=(x,\theta)\ .
\end{equation}
Of course, this property requires the mapping $T$ to be injective.
Then, the final observer is simply
\begin{equation}\label{eq_ObsGen}
\dot z = \Lambda z + Ly\ , \ \dot w = g(w,u)\ ,\ (\hat x,\hat \theta)=T^*(z,w)\ .
\end{equation}}

\subsection{Existence of the mapping $T$}
\label{Sec_Existence}
In \cite{AndrieuPraly_CDC_2004}, it is shown that, in the autonomous case the existence of the mapping $T$, solution of the partial differential equation (PDE) (\ref{EDP}), is obtained for almost all Hurwitz matrices $\Lambda$. \changement{For general controlled nonlinear systems, it is still an open problem to know if it is possible to find a solution.
However, in the particular case of the linear in $x$ controlled system (\ref{eq_ExSyst}), an explicit solution  of the PDE (\ref{EDP}) may be  given.}
\begin{theorem}[Existence of $T$]
Let $r$ be a positive integer.
For all $r$-uplet of negative real numbers $(\lambda_1,\dots,\lambda_r)$ such that, for all $\theta$ in $\Theta$ we have 
\begin{equation}\label{eq_DisjointEV}
\lambda_i\notin\left(\bigcup_{\theta\in\Theta} \sigma\{A(\theta)\} \right) \ ,\ i=1,\dots,r,
\end{equation}
there exists a linear in $x$ function $T:\Re^{n}\times\Theta\times\Re^{r}\rightarrow\Re^{r}$ 
solution to the PDE (\ref{EDP})
with $\Lambda = \diag \{\lambda_1,\dots,\lambda_r\}\ , \ L = \Un_r\,$
and $g:\Re^{r}\times\RR\mapsto \Re^{r}$ defined as $ g(w,u) = \Lambda w + L u $.
\end{theorem}
\begin{proof}
Keeping in mind that the spectrum of $\Lambda$ and $A(\theta)$ are disjoint as required by (\ref{eq_DisjointEV}), let us introduce the matrix $M_i(\theta)$ in $\Re^{1 \times n}$  defined by 
$$
M_i(\theta) =C(\theta) (A(\theta)-\lambda_i I_{n})^{-1}\ 
$$
for all $i$ in $\{1,\dots, r\}$. 
Let $T_i:\Re^{n}\times\Theta\times\Re\rightarrow\Re$ 
be defined as:
\begin{equation}\label{eq_Ti}
T_i(x,\theta,w_i) = M_i(\theta)[x-B(\theta)w_i]\ .
\end{equation}
Let also the vector field $g_i:\Re\times \Re\rightarrow\Re$ be defined as
\begin{equation}\label{eq_gi}
g_i(w_i,u)  = \lambda_i w_i + u\ .
\end{equation}
It can be noticed that $T_i$ is solution to the PDE
$$
\begin{array}{lr}
\dfrac{\partial T_i}{\partial x}(x,\theta,w_i)\!\!&\!\![A(\theta)x+B(\theta)u] +\dfrac{\partial T_i}{\partial w_i}(x,\theta,w_i)g_i(w_i,u)\\
&\\
& = \lambda_i T_i(x,\theta,w_i) + C(\theta)x\ .
\end{array}
$$
Hence, the solution of the PDE (\ref{EDP}) is simply taken as
\begin{equation}\label{eq_T}
T(x,\theta,w)  = \begin{bmatrix}T_1(x,\theta,w_1)&\dots& T_r(x,\theta,w_r)\end{bmatrix}^\top\ .
\end{equation}
This ends the proof.
\end{proof}

\longue{\begin{remark}
Note that in the particular case in which the system is autonomous, the mapping $T$ is given as  
\begin{equation}\label{def_TAut}
T_o(x,\theta) = M(\theta)x\ ,\ M(\theta) = \begin{bmatrix}
M_1(\theta)\\\vdots\\M_r(\theta)
\end{bmatrix} \ .
\end{equation}
This matrix $M(\theta)$ is solution to the following parameterized Sylvester equation
\begin{equation}\label{eq_algLuenberger}
M(\theta)A(\theta) = \Lambda M(\theta) + LC(\theta) \ .
\end{equation}
Hence, taking $r=n$, the well known Luenberger observer introduced in \cite{Luenberger_IEEE_TME_64} in the case of autonomous systems is recovered.
Note however that, here, the injectivity is more involved than in the context of  \cite{Luenberger_IEEE_TME_64} since $\theta$ is unknown.
\end{remark}}
\changement{\begin{remark}
Note that if the set $\Theta$ is bounded, then it is ensured that there exist $(\lambda_i)$'s which satisfy equation (\ref{eq_DisjointEV}).\longue{ Indeed, if $\Theta$ is bounded, then the set $\left(\bigcup_{\theta\in\Theta} \sigma\{A(\theta)\} \right)$ is a bounded set. This can be obtained from the fact that  each eigenvalue $\lambda$ in $\sigma\{A(\theta)\}$ is a zero of the characteristic polynomial:
\begin{equation}\label{CarPoly}
\lambda^n + \mu_1(\theta) \lambda^{n-1} + \cdots + \mu_{n-1}(\theta)\lambda + \mu_n(\theta)=0,
\end{equation}
where $\mu_i(\theta)$ are continuous functions of $\theta$. Boundedness of $\Theta$ together with the continuity of the $\mu_i$'s imply that there is $c>0$ such that $|\mu_i(\theta)|\leq c$ $\forall i\in\{1,\ldots,n\}$, $\forall\theta\in \Theta$. As a consequence if $|\lambda|>1$, we must have 
$$
|\lambda|\leq \sum_{j=1}^n|\mu_j(\theta)| |\lambda|^{1-j}\leq \dfrac{c}{1-1/|\lambda|}
$$
which hence implies that $|\lambda|\leq c+1$. }
\end{remark}}

\subsection{Injectivity of the mapping $T$}
\label{Sec_Injectivity}
As seen in the previous section, it is known that if the following dynamical extension is considered: 
\begin{equation}\label{eq_DynExt}
\dot z = \Lambda z + L y\  , \: \dot w = \Lambda w + L u
\end{equation}
with $z$ in $\Re^{r}$ and $w$ in $\Re^{r}$, then it yields that along the solution of the system defined by (\ref{eq_ExSyst}) and (\ref{eq_DynExt}), equation (\ref{eq:Cvge-in-imT}) is true.
Consequently, $T(x,\theta,w)$ defined in (\ref{eq_Ti})-(\ref{eq_T}) is asymptotically estimated. 
The question that arises is whether this information is sufficient to get the knowledge of $x$ and $\theta$.
This is related to the injectivity property of this mapping.
As shown in \cite{AndrieuPraly_CDC_2004}, in the autonomous case this property is related to the observability of the extended system (\ref{eq_ExSyst}).
With observability, it is sufficient to take $r$  large enough to get injectivity.
%
Here, the same type of result holds if it is assumed  an observability uniform with respect to the input in a specific set.

The following strong observability assumption is made:

\begin{assumption}[Uniform differential injectivity]\label{Ass_UnifDiffInjec}
 There exist two bounded open subsets  $\CR_\theta$ and $\CR_x$ which  closures are respectively in $\Theta$ and $\Re^n$, an integer $r$ and $U_r$ a bounded subset of $\Re^{r-1}$ such that the mapping
 $$
\HH_r(x,\theta,v) = H_r(\theta)x + \sum_{j=1}^{r-1}S^j H_r(\theta)B(\theta)v_{j-1}\ ,
 $$
 with 
 $$ H_r(\theta)=\begin{bmatrix} C(\theta)^\top & (C(\theta)A(\theta))^\top & \cdots & (C(\theta)A(\theta)^{r-1})^\top \end{bmatrix}^\top \ ,$$
 $v=(v_0,\dots,v_{r-2})$ and $S$ is the shift matrix operator such that for all $s=(s_1, \dots, s_r)$, $S\times s = (0,s_1, \dots s_{r-1})$ is injective in $(x,\theta)$, uniformly in $v\in U_r$ and full rank. More precisely, there exists a positive real number $L_\HH$ such that for all $(x,\theta)$ and $(x^*,\theta^*)$ both in $\texttt{Cl}(\mathcal C_\theta)\times \texttt{Cl}(\mathcal C_x)$ and all $v$ in $U_r$
 $$
\left |\HH_r(x^*,\theta^*,v)-\HH_r(x,\theta,v)\right | \geq L_\HH\left |\begin{bmatrix} x-x^*\\ \theta-\theta^*\end{bmatrix}\right |\ .
 $$
 \end{assumption}
The following result establishes an injectivity property for large eigenvalues of the observer.
\begin{theorem}\label{Theo_InjCont}
Assume Assumption \ref{Ass_UnifDiffInjec} holds. Let $u(\cdot)$ be a bounded $C^{r-1}([0,+\infty])$ function with bounded $r-1$ first derivatives, i.e. there exists a positive real number $\mathfrak u$ such that
\begin{equation}\label{eq_Boundu}
|\bar u^{(r-1)}(t)|\leq \mathfrak u\ , \ \forall t\geq 0\ .
\end{equation}
For all $r$-uplet of distinct negative real numbers $(\tilde \lambda_1, \dots, \tilde \lambda_{r})$, for all positive time $\tau$ and for all $w_0$ in $\Re^r$, there exist two positive real numbers $k^*$ and $\bar{L}_T$ such that for all $k>k^*$, the mapping defined in (\ref{eq_Ti})-(\ref{eq_T}) with $\lambda_i=k \tilde \lambda_i$, $i=1, \dots, r$ satisfies the following injectivity property in $\CR=\CR_x\times\CR_\theta$. For all $t_1\geq \tau$, if $\bu^{(r-2)}(t_1)$ is in $U_r$, then for all $(x,\theta)$ and $(x^*,\theta^*)$ in $\CR_x\times\CR_\theta$
the following inequality holds:
\begin{equation}\label{eq_InjTcontr}
\left |T(x,\theta,w(t_1))-T(x^*,\theta^*,w(t_1))\right | \geq
 \frac{\bar{ L}_T}{k^r}\left |\begin{bmatrix} x-x^*\\ \theta-\theta^*\end{bmatrix}\right |
\end{equation}
where $w(\cdot)$ is the solution of the $w$ dynamics in \eqref{eq_DynExt} initiated from $w_0$.
\end{theorem}
The proof of this result is reported in Appendix \ref{Sec_ProofTheoInjCont}.

\begin{remark}
Note that in the case in which the control input is such that for all $t\geq 0$, $\bu^{(r-2)}(t)$ is in $U_r$, the inequality (\ref{eq_InjTcontr}) can be rewritten by removing the time dependency. More precisely, by introducing $\CR_w$ a subset of $\Re^{r}$ defined as
$$
\CR_w = \bigcup_{t\geq t_1}\{w(t)\}, 
$$
The inequality \eqref{eq_InjTcontr} can be restated as follows: 
for all $(x,\theta)$ and $(x^*,\theta^*)$ in $\CR_x\times\CR_\theta$ and all $w$ in $\CR_w$, 
$$
\left |T(x,\theta,w)-T(x^*,\theta^*,w)\right | \geq
 \frac{\bar L_T}{k^r} \left |\begin{bmatrix} x-x^*\\ \theta-\theta^*\end{bmatrix}\right |\ .
$$
\end{remark}

\subsection{Construction of the observer}
\label{Sec_GenConstrObs}
From the existence of an injective function $T$ solution to the PDE (\ref{EDP}), it is possible to formally define a nonlinear Luenberger observer as in equation (\ref{eq_ObsGen}).
Note however that the mapping $T^*$ solution of (\ref{eq:Inverse-Mapping-T}) has to be designed.
Following the approach introduced in \cite{RapaportMaloum_IJRNC_04}, the Mc-Shane  formula can be used (see \cite{Mcshane_BAMS_34} and more recently \cite{MarconiPraly_TAC_08}).

Indeed, assuming we have in hand a function $T$ uniformly injective, then the following proposition holds.
\begin{proposition}
\label{Prop_MacShane}
If there exist bounded open sets $\CR_x$ and $\CR_\theta$ and a set $\CR_w$ such that for all $(x,\theta)$ and $(x^*,\theta^*)$ both in $\CL(\CR_x)\times\CL(\CR_\theta)$ and $w$ in $\CR_w$ 
\begin{equation}\label{eq_InjTcontrGen}
\left |T(x,\theta,w)-T(x^*,\theta^*,w\right) | \geq
 L_T \left | \begin{bmatrix}
 x-x^*\\ 
 \theta-\theta^*\end{bmatrix}
 \right |\ ,
\end{equation}
then the mapping $T^*:\RR^{r}\times \CR_w\rightarrow\RR^n\times\Theta$, 
$T^*(z,w)= \left((T^*_{x_i}(z,w) )_{1\leq i\leq n},  (T^*_{\theta_j}(z,w) )_{1\leq j\leq q}\right)$ 
defined by
\begin{align}
T_{x_i}^*(z,w) &= \inf_{(x,\theta) \in \CL(\CR_x\times\CR_\theta)} \left\{x_i + \frac{1}{L_T}|T(x,\theta,w)-z|\right\}\ , \label{Tstar-X}\\
T_{\theta_j}^*(z,w) &= \inf_{(x,\theta) \in \CL(\CR_x\times\CR_\theta)} \left\{\theta_j + \frac{1}{L_T}|T(x,\theta,w)-z|\right\}, \label{Tstar-theta}
\end{align}
satisfies for all $(z,x,\theta,w)$ in $\RR^{r}\times\CR_x\times\CR_\theta\times \CR_w$
\begin{equation}\label{LipschitzTstar}
\left|T^*(z,w)-\begin{bmatrix}x\\\theta\end{bmatrix}\right| \leq \frac{\sqrt{n+q}}{L_T}\, |z-T(x,\theta,w)|\ .
\end{equation}
\end{proposition}

Note that one of the drawback of the suggested construction for $T^*$ is that this one is based on a minimization algorithm and hence may lead to numerical problems.
An alternative solution has been investigated in \cite{BernardAndrieuPraly_15_HAL_NonLinObsOriCoord} (see also \cite{AndrieuEytardPraly_CDC_14_DynExtInvObs}) to overcome this optimization step but it is still an open question to employ these tools in this context.

Moreover, in Section \ref{Sec_AppliIdentif}, when considering a particular structure of the matrices $A$, $B$ and $C$, an explicit function $T^*$ which does not rely on an optimization is given.

\subsection{Robustness}
\label{Sec_GenRobustess}
In this section, the robustness of the proposed algorithm is investigated.
Note that contrary to most of existing identification algorithms, the convergence result of the current identifier does not rely on LaSalle  invariance principle (as this is the case for instance in \cite{Kreisselmeier_TAC_77,Q_Zhang_2002,NarendraAnnaswamy_Book_89}.
Indeed, considering the function $V:\CR_x\times \RR^r\times \CR_w\rightarrow \RR_+$ 
defined by
\begin{equation}\label{strictLyapFunc}
V(x,\theta, z,w) = |z-T(x,\theta,w)|\ .
\end{equation}
assuming that inequality (\ref{eq_InjTcontrGen}) holds, this implies that
$$
V(x,\theta, z,w) \geq L_T\left|\begin{bmatrix}
x\\
\theta\end{bmatrix}-T^*(z,w)\right|\ .
$$
Along the trajectories of the system, it yields
$$
\dot{\overparen{V(x,\theta, z,w)}}\leq \max_{i=1,..,r}\{\lambda_i\}\, V(x,\theta, z,w)\ ,
$$
with $\lambda_i<0$. 
In other words, $V$ is a strict Lyapunov function associated to the observer.

This allows to give an explicit characterization of the robustness in term of input-to-state stability gain.
Indeed,  consider now the case in which we add three time functions $\delta_x$, $\delta_\theta$ and $\delta_y$ in $\mathcal{L}^\infty_{loc}(\Re_+)$ to the system (\ref{eq_ExSyst}) such that we consider the system
\begin{equation}\label{eq_DistSyst}
\dot x = A(\theta) x + B(\theta)u + \delta_x \ , \ \dot \theta = \delta_\theta \ , \ y = C(\theta) + \delta_y
\end{equation}
where $(\delta_x,\delta_\theta, \delta_y)$ are time functions of appropriate dimensions.

Following the same approach, we consider the observer (\ref{eq_ObsGen}) with the function $T^*$ given in \eqref{Tstar-X}-\eqref{Tstar-theta}.

\begin{proposition}[Robustness]\label{Prop_robust}
Let $\CR_x$, $\CR_\theta$ and $\CR_w$ be three bounded open sets which closure is respectively in $\RR^n$, $\Theta$ and $\RR^r$.
Consider the mapping $T$ given in (\ref{eq_Ti}). Assume that there exist three positive real numbers $L_T$, $L_x$ and $L_\theta$ such that (\ref{eq_InjTcontrGen}) is satisfied
and for all $(x,\theta,w)$ in $\CR_x\times \CR_\theta\times\CR_w$
$$
\left | \frac{\partial T}{\partial x}(x,\theta,w) \right | \leq L_x 
\ ,\ 
\left | \frac{\partial T}{\partial \theta}(x,\theta,w) \right | \leq L_\theta\ ,
$$
then considering  the observer (\ref{eq_ObsGen}) with the function $T^*$ given in \eqref{Tstar-X}-\eqref{Tstar-theta} it yields along the solutions of system (\ref{eq_DistSyst}) the following inequality for all $t$ positive such that $(x(t),\theta(t),w(t))$ is in $\CR_x\times\CR_\theta\times\CR_w$.
\begin{multline}\label{EstimErrorBound}
\left|\begin{bmatrix}\theta(t)-\hat \theta(t)\\
x(t)-\hat x(t)\end{bmatrix}\right|  
\leq 
\\\dfrac{\sqrt{n+q}}{L_T}\exp\left(\max_{i=1\dots,r}\{\lambda_i\} t\right) \left |z(0)-T(x(0),\theta(0),w(0))\right | \\+ \dfrac{\sqrt{n+q}\left(\sup_{s\in[0,t]}\{L_x|\delta_x(s)| + L_\theta |\delta_{\theta}(s)| + \sqrt{r}|\delta_y(s)|\}\right)}{L_T\max_{i=1\dots,r}\{|\lambda_i|\}}.
\end{multline}
\end{proposition}
\longue{\begin{proof}
Note that along the solutions of system (\ref{eq_DistSyst}) and (\ref{eq_ObsGen}), it yields for all $t\geq 0$
$$
\begin{array}{ll}
\dot {\overparen{z-T(x,\theta,w)}}\!\!&\!\!= \Lambda (z-T(x,\theta,w)) \\
&\\
&- \dfrac{\partial T}{\partial x}(x,\theta,w)\delta_x(t) 
- \dfrac{\partial T}{\partial \theta}(x,\theta,w)\delta_\theta(t)
+ \Un_r\delta_y.
\end{array}
$$
The solution of this last equation is given as 
\begin{align*}
z(t)-T(x(t),\theta,w(t))\!&\!=\exp\left(\Lambda t\right)(z(0)-T(x(0),\theta(0),w(0))\\ 
&+\int_0^t\exp\left(\Lambda (t-s)\right)\left(-\frac{\partial T}{\partial x}(x,\theta,w)\delta_x(s)\right. \\ 
&\left. - \frac{\partial T}{\partial \theta}(x,\theta,w)\delta_\theta(s)
+ \Un_r\delta_y\right) ds
\end{align*}
Hence, the norm $\left |z(t)-T(x(t),\theta(t),w(t))\right |$ is upper bounded as
\begin{align*}
\left|z(t)-T(x(t),\right.\!&\!\left.\theta(t),w(t))\right|\leq\int_0^t \exp\left(\max_{i=1\dots,r}\{\lambda_i\} (t-s)\right)ds \\
&\times\sup_{s\in[0,t]}\{L_x|\delta_x(s)| + L_\theta |\delta_{\theta}(s)| + \sqrt{r}|\delta_y(s)|\}\\
&+\exp\left(\max_{i=1\dots,r}\{\lambda_i\} t\right) \left |z(0)-T(x(0),\theta(0),w(0)\right | \\ 
&\leq \dfrac{\sup_{s\in[0,t]}\{L_x|\delta_x(s)| + L_\theta |\delta_{\theta}(s)| + \sqrt{r}|\delta_y(s)|\}}{\max_{i=1\dots,r}\{|\lambda_i|\}}\\
&+\exp\left(\max_{i=1\dots,r}\{\lambda_i\} t\right) \left |z(0)-T(x(0),\theta(0),w(0)\right |
\end{align*}
Consequently with the function $T^*$ defined in \eqref{Tstar-X}-\eqref{Tstar-theta}, it yields from Proposition \ref{Prop_MacShane} equation  (\ref{LipschitzTstar}) that the result holds.
\end{proof}
}
\courte{\changement{
The proof is based on the use of the strict Lyapunov function (\ref{strictLyapFunc}). This one can be found in \cite{afri:hal-01232747}.
\begin{remark}
Following the proof of Theorem \ref{Theo_InjCont}, it can be seen that  $L_T = \frac{\bar L_T}{k^r}$. Moreover, from (\ref{EstimErrorBound}), it can be checked that increasing the speed of convergence (by increasing the eigenvalues factor $k$) of the observer has the consequence of reducing its robustness to output and state perturbations.
\end{remark}
}}
\longue{\begin{remark}
It may be interesting to see how the constants $L_T$, $L_x$ and $L_\theta$ behave when the eigenvalues of the observer are multiplied  by a positive real number $k$.
Following the proof of Theorem \ref{Theo_InjCont}, it can be seen that  $L_T = \frac{\bar L_T}{k^r}$. Moreover, it can be checked that the following estimation can be made: 
$$
\left | \frac{\partial T}{\partial x}(x,\theta,w) \right | \leq \frac{C_x}{k} 
\ ,\ 
\left | \frac{\partial T}{\partial \theta}(x,\theta,w) \right | \leq \frac{C_\theta}{k}\ 
$$
with $C_\theta$ and $C_x$ denoting some constant numbers. 
As a consequence, the previous bound becomes
\begin{multline}
\left|\begin{bmatrix}\theta(t)-\hat \theta(t)\\
x(t)-\hat x(t)\end{bmatrix}\right| \leq \\
\dfrac{k^r\sqrt{n+q}}{\bar L_T}\exp\left(k\max_{i=1\dots,r}\{\dlambda_i\} t\right) \left |z(0)-T(x(0),\theta(0),w(0)\right | \\
+ \dfrac{k^{r}\sqrt{n+q}\sup_{s\in[0,t]}\left\{\dfrac{C_x}{k}|\delta_x(s)| + \dfrac{C_\theta}{k} |\delta_{\theta}(s)| + \sqrt{r}|\delta_y(s)|\right\}}{\bar L_T\max_{i=1\dots,r}\{|\dlambda_i|\}}\ .
\end{multline}
From this estimate, we conclude that increasing the speed of convergence (by increasing the eigenvalues factor $k$) of the observer has the consequence of reducing its robustness to output and state perturbations.
\end{remark}}

\section{Application to system identification problems}
\label{Sec_AppliIdentif}

\subsection{Considered realization}
\label{sec_ConsideredSystIdentif}
In the previous section, it has been shown that based on a differential observability assumption and its associated set of \text{good} inputs $U_r$, it is possible to design a robust observer which reconstructs the state and the unknown parameters of a linear system in the form \eqref{Sys} as long as the input remains in $U_r$.

Note however that this observer relies on the construction of a mapping $T^*$ given in \eqref{Tstar-X}-\eqref{Tstar-theta} which requires a nonlinear (and probably non convex) optimization.
In this section, a particular canonical structure for system \eqref{Sys} is considered.
This allows to give an explicit construction of a mapping $T^*$ left inverse of $T$. Moreover, it allows to give a complete characterization of the dimension of the observer and the class of inputs which guarantee that the differential observability property (i.e. Assumption \ref{Ass_UnifDiffInjec}) holds.

The considered particular canonical structure for the matrix-valued functions $A$, $B$, $C$ is given as follows.
\begin{equation}\label{eq:canonical-form}
A(\theta)=\left[\!\!\begin{array}{cc}\begin{array}{c}\theta_a \end{array}\!\! & \begin{array}{:c}I_{n-1}\\\hdashline 0\end{array}\end{array}\!\! \right] \ , ~~~~ B(\theta)=\theta_b\ , ~~~~C=e_1^\top
\end{equation}
where 
$$ e_1=\begin{bmatrix} 1&0&\cdots&0 \end{bmatrix}^\top\in \RR^{n \times 1}\ , ~~~~ \theta=\begin{bmatrix}\theta_a^\top&\theta_b^\top\end{bmatrix}^\top\in \RR^{2n \times 1}$$
Note that assuming the structures \eqref{eq:canonical-form} for $A,B,C$ is without loss of generality: any input-output behavior of a linear SISO system can be described with a model of  this structure \changement{(maybe after a linear change of coordinates)}. Such a realization  is observable  for any vector $\theta$. 

The interest of this structure is twofold:
\begin{enumerate}
\item it is possible to select $r$ and to characterize the class of input such that Assumption \ref{Ass_UnifDiffInjec} is satisfied.
\item it is possible to give explicitly a candidate for the mapping $T^*$ which allows us to define a complete algorithm.
\end{enumerate}

The following two subsections are devoted to addressing these two points. The complete identification algorithm is given at the end of this section.

\subsection{Input generation in order to satisfy the Assumption \ref{Ass_UnifDiffInjec}}
\label{Sec_GenInput}

It is usual that in adaptive control and in identification problem the class of input considered is sufficiently exciting. This means that the signal has to be composed of a sufficiently large number of frequencies such that some integrals are positive definite. 
The characterization of a \textit{good} input is now well understood for discrete time systems.
For instance, as mentioned in \cite{Goodwin_Book_14_Adaptive}, a sequence of input $(u(k))_{k\in\NN}$ is sufficiently rich of order $p$ if there exist $m\in \mathbb{N}$ and $\rho>0$ such that the following inequality holds for all integer $k$
$$
\sum_{i=k}^{k+m}\begin{bmatrix}
u(i)& 
\dots &
u(i+p-1)
\end{bmatrix}^\top
\begin{bmatrix}
u(i)& 
\dots &
u(i+p-1)
\end{bmatrix}\geq \rho I_p.
$$
There have been some attempts to extend this assumption to continuous time systems (see \cite{Janecki_87_SCL_persistency} or \cite{ShimkinFeuer_87_SCLpersistency}).
In the context of this paper, the approach is different.
The assumption we make on the input is that sufficient information is obtained from its successive time derivatives. To be more precise,  given an integer $r$ and a vector $v=(v_0,\dots,v_{2r})$ in $\Re^{2r+1}$  we introduce $\mat_r(v)$ the $(r+1)\times (r+1)$ (Hankel) real matrix defined as
\begin{equation}
\mat_r(v) = \begin{bmatrix}
v_0 & v_1 & \dots & v_{r}\\
v_1 & v_2 & \dots & v_{r+1}\\
\vdots&\vdots&\ddots&\vdots\\
v_{r} & v_{r+1} & \dots & v_{2r}
\end{bmatrix}
\end{equation}
With this notation, we can now define the notion of \textit{differentially exciting inputs}.
\begin{definition}[Differentially exciting function]
\label{def:diff-exciting}
A  $C^{2r}$ function $u:\RR\mapsto\RR$ is said to be differentially exciting of order $r$ at time $t$ if the matrix $\mat_r(\bar u^{(2r)}(t))$ is invertible.
\end{definition}
As  it will be shown in the following proposition, there is a link between this property and the property of persistency of excitation for continuous time system (as introduced for instance in \cite{ShimkinFeuer_87_SCLpersistency}). 
\begin{proposition}[Link with persistency of excitation]\label{Prop_LinkPersis}
Let $u:\RR\rightarrow\RR$ be a $C^{2r}$ function which is differentially exciting of order $r$ at time $t$. Then there exist two positive real numbers $\epsilon(t)$ and $\rho(t)$ such that 
\begin{equation}\label{eq_StandardDiffExc}
\int_{t}^{t+\epsilon(t)}
\bar{u}^{(r)}(s)
\big(\bar {u}^{(r)}(s)\big)^\top
ds 
\geq \rho(t) I\ .
\end{equation}
\end{proposition}
\courte{
\changement{
The proof of this proposition is available in \cite{afri:hal-01232747}.
}}

\longue
{
The proof of this proposition is given in Appendix \ref{Sec_ProofPropLinkPersis}.
\begin{remark}
As seen in the proof of the proposition, when $\left (\mat_r(\bar u^{(2r)}(t))\right )^\top \mat_r(\bar u^{(2r)}(t)) \geq \rho_u I$ with $\rho_u$ independent of $t$ and when the first $2r+1$ derivatives of $u$ are bounded for all $t$, $\epsilon$ may not depend on $t$.
This implies that inequality (\ref{eq_StandardDiffExc}) can be made uniform in time.
\end{remark}
}
The interest we have in inputs satisfying the differential exciting property is that if at each time this property is satisfied for $r=2n$, then  the mapping $\HH_{4n-1}$ satisfies Assumption \ref{Ass_UnifDiffInjec} when we restrict attention to sets $\Theta$ of coefficients $\theta=[\thetaa, \thetab]$ for which the couple $(A(\thetaa), B(\thetab))$ is controllable.
\begin{proposition}\label{Prop_AssIndentif}
Let $\CR_x$ be a bounded open set in $\Re^n$.
Let $\CR_\theta$ be a bounded open set with closure  in $\Theta$. Let $(A(\cdot),B(\cdot),C(\cdot))$ have the  structure (\ref{eq:canonical-form}) and be such that for all $\theta=(\thetaa,\thetab)$ in $\CL(\CR_\theta)$ the couple 
$(A(\thetaa), B(\thetab))$ is controllable. 
Let $U_{4n}$ be a compact subset of $\Re^{4n-1}$ such that for all $v=(v_0,\dots,v_{4n-2})$ in $U_{4n}$ the matrix $\mat_{n-1}(v)$ is invertible.
Then Assumption \ref{Ass_UnifDiffInjec} is satisfied.
More precisely there exists a positive real number $L_\HH$ such that for all $(x,\theta)$ and $(x^*,\theta^*)$ both in $\texttt{Cl}(\mathcal C_\theta)\times \texttt{Cl}(\mathcal C_x)$ and all $v$ in $U_{4n}$
$$
\left |\HH_{4n-1}(x^*,\theta^*,v)-\HH_{4n-1}(x,\theta,v)\right | \geq L_\HH		\left |\begin{bmatrix} x-x^*\\ \theta-\theta^*\end{bmatrix}\right |\ .
 $$
\end{proposition}

The proof of this proposition is given in Appendix \ref{Sec_ProofPropAssIndetif}.

A natural question that arises from the former Proposition is whether or not it is possible to generate an input which satisfies the differentially exciting property. As shown in the following proposition, inputs having such property may be easily generated by observable and conservative linear systems. 
\begin{lemma}[Generation of differentially exciting input]
\label{lem:exciting-inputs}
Consider the linear system
\begin{equation}\label{eq:example-exciting-inputs}
\dot v = Jv\ ,\ u=Kv\quad v(0)=v_0
\end{equation}
 with $v$ in $\RR^{2r}$ and $J$ being  an invertible skew adjoint matrix with distinct eigenvalues
 and $K$ a matrix such that the couple $(J,K)$ is observable.
Then there exists $v_0$ in $\RR^{2r}$ such that $u$ is differentially exciting of order $2r-1$ for all time.
\end{lemma}
\courte{
\changement{
The proof of this Lemma can be found in \cite{afri:hal-01232747}.
}}

\longue{
\begin{proof}
Direct calculations show that
$$\mat_r(\bar u^{(4r-2)}(t))=\begin{bmatrix}K\\ KJ\\ \vdots\\ KJ^{2r-1}\end{bmatrix}\begin{bmatrix}v(t) &Jv(t) & \cdots & J^{2r-1}v(t) \end{bmatrix}.$$
System  \eqref{eq:example-exciting-inputs} being observable, invertibility of the matrix $\mat_r(\bar u^{(4r-2)}(t))$ is obtained if the second matrix is full rank for some $v_0$. 
To this end, note $J$ being skew adjoint and invertible there exist $\omega_i$, $i=1,\ldots,r$,  real positive and distinct numbers such that $J$ can be written (in some specific coordinates) in the form 
$$J=\texttt{Diag}\{S(\omega_1),\cdots,S(\omega_r)\}\in \mathbb{R}^{2r\times 2r}\ ,$$
where 
$$S(\omega_i)=\begin{bmatrix}0&\omega_i\\-\omega_i&0 \end{bmatrix} \ .$$
The minimal polynomial of such a matrix $J$  has degree equal to its dimension $2r$. As a consequence, there exists a nonzero vector $v_0$ such that $\begin{bmatrix}
v_0 & Jv_0 &\dots J^{2r-1}v_0
\end{bmatrix}$
is non singular. For example, it can be verified that $v_0=\begin{bmatrix}0&1&0&1&\cdots& 0&1 \end{bmatrix}^\top$ (i.e., with one entry out of two  equal to $1$) fulfills the condition. 
Let the initial state $v_0$ of \eqref{eq:example-exciting-inputs} be selected so as to satisfy this condition. Then the state  trajectory of \eqref{eq:example-exciting-inputs} is defined by 
$v(t) = e^{Jt}v_0 $ 
with $$e^{Jt}=\texttt{Diag}\{e^{S(\omega_1)t},\cdot\cdot,e^{S(\omega_r)t}\}\ , e^{S(\omega_i)t}= \begin{bmatrix}\!\cos(\omega_i t)\!&\!\sin(\omega_it)\!\\\!-\sin(\omega_it)\!&\!\cos(\omega_i t)\!\end{bmatrix} $$
%
We then claim that for any $t$, $\begin{bmatrix}
v(t) & Jv(t) &\dots J^{2r-1}v(t)
\end{bmatrix}$
is also non singular. To see this, suppose for contradiction that the matrix in question is singular. Then there is a nonzero polynomial $p(z)$ of degree less than  $2r$ such that
$p(J)e^{Jt}v_0=0$. Since $e^{Jt}$ commutes with any polynomial of $J$, we have $e^{Jt}p(J)v_0=0$ which in turn implies that $p(J)v_0=0$ because $e^{Jt}$ is invertible. But the last equality  contradicts the assumption made on $v_0$. 
\end{proof}
}
Lemma \ref{lem:exciting-inputs}  can be employed to  select signals that fulfill the differentially exciting property. For example it follows from this lemma that a multisine signal of the form 
 $u(t)=\sum_{i=1}^r\alpha_i\sin(\omega_i t) $
where $\alpha_i\neq 0\:  \forall i$, $\omega_i\neq 0\: \forall i$ and $\omega_i\neq \omega_k$ for $i\neq k$, is differentially exciting of order $2r-1$. Indeed, the multisine signal  corresponds to the situation when 
$K=\begin{bmatrix}\bar{\alpha}_1 & \cdots & \bar{\alpha}_r \end{bmatrix}$, $\bar{\alpha}_i=\begin{bmatrix}\alpha_i & 0 \end{bmatrix}$,  $v_0=\begin{bmatrix}0&1&0&1&\cdots& 0&1 \end{bmatrix}^\top$ and $J$ defined as \longue{in the proof of Lemma \ref{lem:exciting-inputs}.}
\courte{$$J=\texttt{Diag}\{S(\omega_1),\cdots,S(\omega_r)\}\in \mathbb{R}^{2r\times 2r}\ ,~~ S(\omega_i)=\begin{bmatrix}0&\omega_i\\-\omega_i&0 \end{bmatrix} \ .$$}

\subsection{Explicit candidate for the mapping $T^*$}
\label{Sec_ExplicitTstar}
\changement{Another interest of the canonical structure given in \eqref{eq:canonical-form} is that it leads to a simple expression of the left inverse  $T^*$ of the mapping $T$.
Indeed, as shown in the Appendix, it is possible to show that the function $T$ satisfies the following equality for all $(x,\theta,w)$,
\begin{equation}\label{def_Pi}
T_i(x,\theta,w)=\underbrace{\left[ \begin{array}{ccc}  V_i^\top & T_i(x,\theta,w) V_i^\top &-w_iV_i^\top\end{array}\right]}_{P_i(T_i,w_i)}
\begin{bmatrix}x \\ \theta_a \\ \theta_b\end{bmatrix}  
\end{equation}
with $V_{i} = -\begin{bmatrix}
\dfrac{1}{\lambda_i}& \dots & \dfrac{1}{\lambda_i^n}
\end{bmatrix}^\top$.
The former equality can be rewritten
$$
T(x,\theta) = P(z,w)\begin{bmatrix}x \\ \theta_a \\ \theta_b\end{bmatrix} 
$$
with $z=T(x,\theta)$ and $P(z,w)=\begin{bmatrix}P_1(z_1,w_1)^\top & \cdots &P_{r}(z_{r},w_{r})^\top  \end{bmatrix}^\top$.
From this, we see that a natural candidate for a left inverse of $T$ is simply to apply a left inverse to the matrix $P$. This left inverse does not require any optimization step. Note however, that there may exist some point 
$(z,w)$ in which this matrix is not full rank. This implies that the left inverse obtained following this route may not be continuous and this is the price to pay to get a constructive solution. 
However, since it is known that $z$ converges asymptotically to $\im T$, it may be shown that after a transient period $z$ reaches the set in which $P$ becomes left invertible.
\longue{A solution to avoid discontinuity has been deeply investigated in \cite{PralyEtAl_MTNS_06_ObsOscillator}
considering an autonomous second order system with only one parameter.
It is an open question to know if these tools could be applied in the current context.}
The result which is obtained is the following.}
\begin{proposition}[Explicit $T^*$]
\label{Prop_ExplTinv}
Let $\CR_x$, $\CR_\theta$ and $\CR_w$ be three bounded open sets which closure are respectively in $\RR^n$, $\RR^{2n}$ and $\RR^r$.
Let $r$ be a positive integer and a $r$-uplet of negative real numbers $(\lambda_1,\dots,\lambda_r)$ such that
\eqref{eq_DisjointEV} holds. Consider the associated mapping 
$T:\CR_x\times \CR_\theta\times \CR_w\rightarrow\RR^r$
given in (\ref{eq_Ti}) and assume that there exists a positive real number $L_T$, such that (\ref{eq_InjTcontrGen}) is satisfied for all $(x,\theta)$ and $(x^*,\theta^*)$ both in $\CL(\CR_x)\times\CL(\CR_\theta)$ and $w$ in $\CR_w$.
Then there exist three positive real numbers $p_{\min}$, $\epsilon_T$ and $L_{T^*}$ such that the function 
\begin{equation}\label{ExpTInv}
T^*(z,w) \!=\! \left\{\!\begin{array}{ll}
( P(z,w)^\top P(z,w))^{-1}P(z,w)^\top z \!&\! \text{if  } P^\top P\geq p_{\min}I_{3n}
\\
0  \!&\! \text{elsewhere}
\end{array}
\right.
\end{equation}
is well defined and satisfies for all $(z,w,x,\theta)$ such that $|z-T(x,\theta,w)|\leq \epsilon_T$ the following inequality
\begin{equation}\label{Est_TinvExp}
\left|T^*(z,w)-\begin{bmatrix}x\\\theta\end{bmatrix}\right| \leq L_{T^*}\, |z-T(x,\theta,w)|\ .
\end{equation}
\end{proposition}

The proof of Proposition \ref{Prop_ExplTinv} is given in Appendix \ref{Sec_ProofPropExplTinv}.


Employing the results obtained so far, it is possible now to derive a complete algorithm and criterion for convergence of the proposed estimation scheme.
\begin{theorem}\label{Theo_IdentifGen}
Consider the system with $A$, $B$, $C$ defined in (\ref{eq:canonical-form}) and with the input $u$ defined as
$$
\dot v(t) = Jv(t)\ ,\ u(t)=Kv(t)\quad v(0)=v_0
$$
Let $\CR_x$ be a bounded open set in $\Re^n$.
Let $\CR_\theta$ be a bounded open set which closure is in $\Theta$ and such that for all $\theta=(\thetaa,\thetab)$ in $\CL(\CR_\theta)$ the couple 
$(A(\thetaa), B(\thetab))$ is controllable. 
Given $(\dlambda_1, \dots, \dlambda_{r})$ with $r=4n-1$, there exists  $k^*>0$ such that for all $k>k^*$, the observer  (\ref{eq_ObsGen}) with $\lambda_i = k\dlambda_i$ with the function $T^*$ given in (\ref{ExpTInv})
 yields the following property. For all solution $(x(t),\theta)$ which remains in $\CR_x\times\CR_\theta$, it yields
$$
\lim_{t\rightarrow +\infty}|x(t)-\hat x(t)|=0\ ,\ \lim_{t\rightarrow +\infty}|\theta-\hat \theta(t)|=0\ .
$$
\end{theorem}
\changement{
\begin{proof}
Theorem 3 is a direct consequence of Propositions 4, 5 and Lemma 1. 
\end{proof}
}
\longue{\subsection{Numerical illustration}
\label{sec_NumIllustration}
In this part we show via simulation the performances and robustness of the observer (\ref{eq_DynExt})-(\ref{ExpTInv}) in the presence of an output  noise.
Let us select an  controllable and observable third order  system of the class (\ref{Sys}) where matrices $A$, $B$ and $C$ are given as
$$\begin{array}{c}
A=\left[ \begin{array}{rrr} -2.31 & -0.17&-0.16\\-0.17 &-1.02&0.04\\
-0.15&0.04&-0.26\end{array}\right];\ B=\left[ \begin{array}{c}0\\0.88\\0\end{array}\right];\\
\\
C=\left[\begin{array}{rrr} 1.18&-0.78&-0.96\end{array}\right].
\end{array}$$
Since this system is observable it admits an canonical representation of the form \eqref{eq:canonical-form} with matrices $\hat{A},\hat{B},\hat{C}$ given by 
$$\begin{array}{c}
\hat{A}(\hat{\theta})=\left[\begin{array}{ccc}-\hat{\theta}_{a1}& 1 & 0\\-\hat{\theta}_{a2} & 0 & 1\\ -\hat{\theta}_{a3}& 0&0\end{array}\right];\ \hat{B}(\hat{\theta})=\begin{bmatrix}\hat{\theta}_{b1}  \\ \hat{\theta}_{b2} \\ \hat{\theta}_{b3} \end{bmatrix} ;\\
\\
\hat{C}(\hat{\theta})=\left[\begin{array}{ccc} 1 &0 & 0\end{array}\right] \end{array}$$
We set in Table \ref{tab1} the observer configuration and necessary initial points to run a simulation with the Matlab software. 
\begin{table}[h]
\renewcommand{\arraystretch}{1.4}
\centering{
\begin{tabular}{|l|}
\hline
$r=4n-1$ ; $n=3$ ; $x(0)=0$ ; $z(0)=w(0)=0$ \\ 
\hline 
$\hat{\theta}_a(0)=\hat{\theta}_b(0)=0$; $\hat{x}(0)=0$  \\ 
\hline 
$\Lambda=k\left(Diag([0.1~0.2~0.3~0.4~0.5~0.6~0.7~0.8~0.9~1~1.1])\right)$\\ 
\hline 
The input $u(t)$ is a sum of sin signals of $4n-1$ distinct frequencies\\ 
\hline
\end{tabular}
}
\caption{System configuration.}\label{tab1}
\end{table}

The results of simulation are given in Fig.\ref{fig1} and Fig.\ref{fig3}. We can see from Fig.\ref{fig1} that the estimated system eigenvalues $\sigma\{\hat{A}(\hat{\theta})\}$ (which are invariant through a similar transformation) converge to the real system eigenvalues. Moreover, the speed of convergence is proportional to the gain $k$ but on the other side the output noise ($40dB$) effect is also proportional to $k$. 
As a consequence, a trade-off must be found between speed of convergence and robustness. This is completely in line with the result of Proposition 2. Another invariant parameter through a similar transformation is the relative error 
$$Err(\hat{\theta})=\dfrac{\|O_nB-\hat{O}_n(\hat{\theta})\hat{B}(\hat{\theta})\|}{\|O_nB\|},~~\hat{O}_n(\hat{\theta})=\begin{bmatrix}\hat{C} \\ \hat{C}\hat{A}(\hat{\theta}) \\ \vdots \\ \hat{C}\hat{A}^{n-1}(\hat{\theta})\end{bmatrix}$$ 
presented in Fig.\ref{fig3} which gives consistent results with those of Fig.\ref{fig1}. 
\begin{figure}
\begin{center}
\def\svgwidth{\columnwidth}
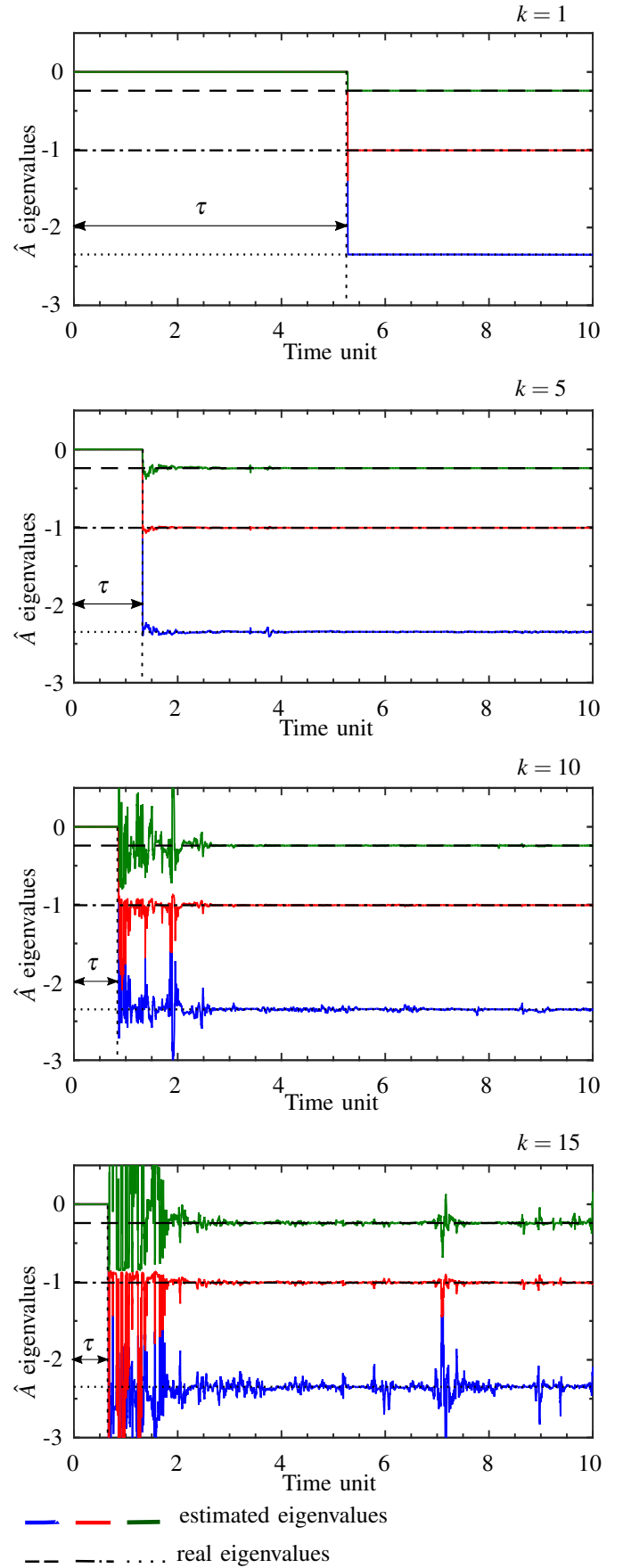
\caption{Convergence of Matrix $\hat{A}$ eigenvalues to the target (Matrix A eigenvalues) in presence of added output noise of $40 dB$ and for various values of the observer gain}\label{fig1}
\end{center}
\end{figure}

\begin{figure}
\begin{center}
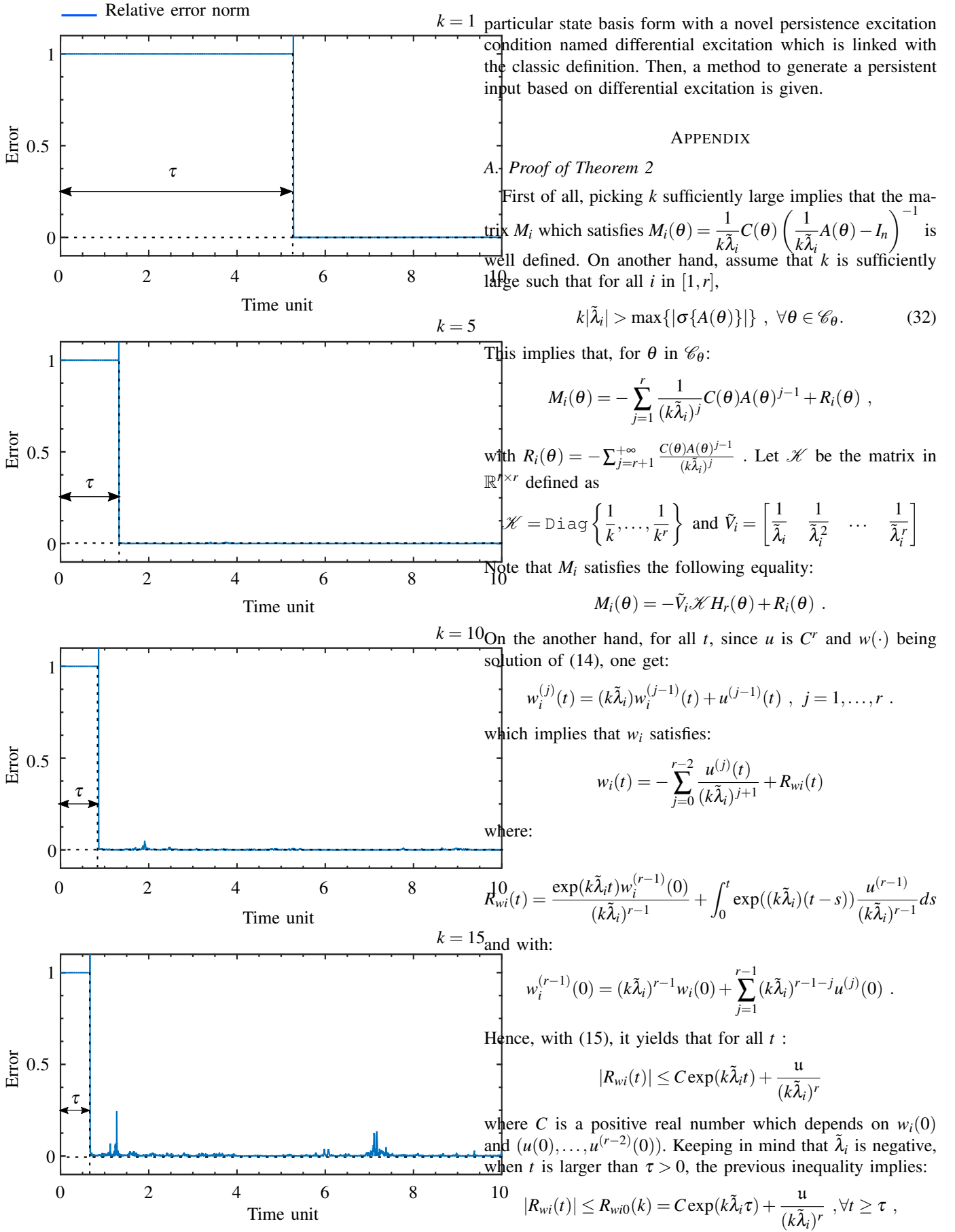
\caption{Evolution of Markov parameters in presence of added output noise of $40 dB$ and for $k=1,5,10,15$.}\label{fig3}
\end{center}
\end{figure}

}
\changement{
\section{Conclusions}
The design of a nonlinear Luenberger observer to estimate the state and the unknown parameters of a parameterized linear system was studied here.
In a first part of the study, a Luenberger observer was shown to exist. This result is obtained from the injectivity property of a certain mapping. In a second part, a simple identification algorithm was given for a particular state basis form with a novel persistence excitation condition named differential excitation which is linked with the classic definition. Then, a method to generate a persistent input based on differential excitation is given.
} 

\appendix

\subsection{Proof of Theorem \ref{Theo_InjCont}}
\label{Sec_ProofTheoInjCont}

First of all, picking $k$ sufficiently large implies that the matrix $M_i$ which satisfies
$M_i(\theta) = \dfrac{1}{k\tilde{\lambda}_i}C(\theta)\left( \dfrac{1}{k\tilde{\lambda}_i}A(\theta)-I_n\right)^{-1}$
is well defined.
On another hand, assume that $k$ is sufficiently large such that for all $i$ in $[1,r]$, 
\begin{equation}
k| \dlambda_i |  > \max\{|\sigma\{A(\theta)\}|\}\ ,\ \forall \theta\in\CR_\theta.
\end{equation}
This implies that, for $\theta$ in $\CR_\theta$:
\begin{equation*}
M_i(\theta) = -\sum_{j=1}^{r} \frac{1}{(k\dlambda_i)^j}C(\theta)A(\theta)^{j-1} + R_i(\theta) \ ,
\end{equation*}
with
$
R_i(\theta) = -\sum_{j=r+1}^{+\infty} \frac{C(\theta)A(\theta)^{j-1}}{(k\dlambda_i)^j}\ .
$
Let $\KR$ be the matrix in $\Re^{r\times r}$ defined as 
$$
\KR = \diag\left \{\dfrac{1}{k},\dots, \dfrac{1}{k^r}\right \} \text{ and }
\tilde V_i = \begin{bmatrix} \dfrac{1}{\tilde \lambda_i} & \dfrac{1}{\tilde \lambda_i^2} & \dots & \dfrac{1}{\tilde \lambda_i^{r}} \end{bmatrix}
$$
Note that $M_i$ satisfies the following equality:
$$
M_i(\theta) = - \tilde{V}_i\KR H_r(\theta) +R_i(\theta)\ .
$$
On the another hand, for all $t$, since $u$ is $C^r$ and $w(\cdot)$ being solution of \eqref{eq_DynExt}, one get:
$$
w_i^{(j)}(t)= (k\dlambda_i) w_i^{(j-1)}(t)+ u^{(j-1)}(t)\ ,\ j=1,\dots,r\ .
$$
which implies that $w_i$ satisfies:
$$
w_i(t) = - \sum_{j=0}^{r-2}\frac{u^{(j)}(t)}{(k\dlambda_i)^{j+1}} + R_{wi}(t)
$$
where:

\begin{align*}
R_{wi}(t)
&= \dfrac{\exp(k\dlambda_i t)w_i^{(r-1)}(0) }{(k\dlambda_i)^{r-1}}+ \int_0^t\exp((k\dlambda_i)(t-s))\frac{u^{(r-1)} }{(k\dlambda_i)^{r-1}}ds
\end{align*}
and with:
$$
w_i^{(r-1)}(0)=(k\dlambda_i)^{r-1} w_i(0) +
\sum_{j=1}^{r-1}(k\dlambda_i)^{r-1-j}u^{(j)}(0)\ .
$$
Hence, with (\ref{eq_Boundu}), it yields that for all $t$
:
$$
|R_{wi}(t)|\leq C\exp(k\dlambda_i t) + \frac{\mathfrak u}{(k\dlambda_i)^r}
$$
where $C$ is a positive real number which depends on $w_i(0)$ and $(u(0), \dots, u^{(r-2)}(0))$.
Keeping in mind that $\dlambda_i$ is negative, when $t$ is larger than $\tau>0$, the previous inequality implies:
$$
\left |R_{wi}(t)\right | \leq R_{wi0}(k) = 
C\exp(k\dlambda_i \tau) + \frac{\mathfrak u}{(k\dlambda_i)^r}
\ , \forall t\geq \tau\ ,
$$
where $R_{wi0}$ depends on $k$ but not on $t$.

By collecting terms of higher order in $\frac{1}{k}$ in a function denoted $R_{MBi}$, it yields the following equality.
\begin{equation*}
M_i(\theta) B(\theta)w_i(t) = \tilde V_i \KR \sum_{j=1}^{r-1}S^j H_r(\theta)B(\theta)u^{(j-1)}(t) + R_{MBi}(\theta,t),
\end{equation*}
and 
with

\begin{align*}
R_{MBi}(\theta,t) &= \sum_{j=1}^r \sum_{\ell =r-j}^{r-2}\dfrac{C(\theta)A(\theta)^{j-1}u^{(\ell)}(t)}{(k\dlambda)^{j+\ell+1}}\\
& - R_{wi}(t) \tilde{V}_i\KR H_r(\theta) - R_i(\theta)\sum_{j=0}^{r-2}\dfrac{u^{(j)}(t)}{(k\dlambda_i)^{j+1}}\ .
\end{align*}

Using the fact that $\CR_\theta$ and $u^{(j)}(t)$ are bounded, it yields the existence of two positive real numbers $C_0$ and $C_1$ such that for all $t\geq \tau$ :
$$
\left |R_{MBi}(\theta,t)\right | \leq \frac{C_{0}}{k^{r+1}} 
\ ,\ 
\left |\frac{\partial R_{MBi}}{\partial \theta}(\theta,t)\right |\leq \frac{C_{1}}{k^{r+1}} 
\ .
$$
Finally with (\ref{eq_Ti})
$$
T_i(x,\theta,w_i(t)) = \tilde V_i \KR \HH_r(x,\theta,\bar{u}^{(r-2)}(t)) + R_{Ti}(x,\theta,t)
$$
with
$$
R_{Ti}(x,\theta,t) = R_i(\theta)x + R_{MBi}(\theta,t)
$$
By denoting $R_T(x,\theta, t) = (R_{T1}(x,\theta,t), \dots R_{Tr}(x,\theta,t))$, this implies:
\begin{equation}
\label{eq_DecompT}
T(x,\theta,w_i(t)) = \tilde {\mathcal V} \KR \HH_r(x,\theta,\bu^{(r-2)}(t))+  R_{T}(x,\theta,t)\ ,
\end{equation}
where $\tilde {\mathcal V}$ in $\Re^{r\times r}$ is the Vandermonde matrix defined as:
$$
\tilde {\mathcal V}= \begin{bmatrix}
\frac{1}{\dlambda_1} & \cdot \cdot & \frac{1}{\dlambda_1^r}\\
:&&:\\
\frac{1}{\dlambda_r} & \cdot \cdot & \frac{1}{\dlambda_r^r}
\end{bmatrix}\ .
$$
Note that $R_T$ is a $C^1$ function and it is possible to find two positive real numbers $C_{T0}$ and $C_{T1}$ such that for all $(x,\theta)$ in $\CR_x\times\CR_\theta$ and $t\geq \tau$:
\begin{align*}
\left |\frac{\partial R_{T}(x,\theta,t)}{\partial x}\right |\leq \frac{C_{T0}}{k^{r+1}}\ ,\ 
\left|\frac{\partial R_{T}(x,\theta,t)}{\partial \theta}\right |\leq \frac{C_{T1}}{k^{r+1}}
\ .
\end{align*}
Hence the mapping $R_T$ is globally Lipschitz with a Lipschitz constant in $o \left (\frac{1}{k^{r}}\right )$. Hence, it is possible to find $k_0$ such that for all $k\geq k_0$ and all quadruples $(x,x^*,\theta,\theta^*)$ in $\CR_x^2\times\CR_\theta^2$, for all $t\geq \tau$:
\begin{equation}\label{eq_RT}
\left |R_T(x,\theta,t)-R_T(x^*,\theta^*,t)\right |
\leq \frac{L}{2\left |\tilde{\mathcal V}^{-1}\right |k^r}  \left |\begin{bmatrix} x-x^*\\ \theta-\theta^*\end{bmatrix}\right | \ .
\end{equation}
It can be shown that the result holds with this value of $k_0$. Indeed, employing \eqref{eq_DecompT}, it yields that, for all $t$:

\begin{align*}
\left|T(x,\theta,w(t))-T(x^*,\theta^*,w(t))\right|\geq - \left|R_T(x,\theta,t)-R_T(x^*,\theta^*,t)\right|\\
+\left |\tilde {\mathcal V}\KR \left (\HH_r(x,\theta,\bu^{(r-2)}(t))-\HH_r(x^*,\theta^*,\bu^{(r-2)}(t))\right)\right|\ ,\\
\\
\left|T(x,\theta,w(t))-T(x^*,\theta^*,w(t))\right|\geq - \left |R_T(x,\theta,t)-R_T(x^*,\theta^*,t)\right |\\
+ \frac{\left|\HH_r(x,\theta,\bu^{(r-2)}(t))-\HH_r(x^*,\theta^*,\bu^{(r-2)}(t))\right|}
{ |\tilde {\mathcal V}^{-1}||\KR^{-1}|}\ .
\end{align*}

Consider now $t_1\geq \tau$, the last term of the previous inequality can be lower-bounded by (\ref{eq_RT}).
Moreover, if $u(t_1)\dots, u^{(r-2)}(t_1)$ is in $U_r$, the other term can be lower-bounded based on Assumption \ref{Ass_UnifDiffInjec} and the result follows.
\longue{

\subsection{Proof of Proposition \ref{Prop_LinkPersis}}
\label{Sec_ProofPropLinkPersis}
Given an integer $\ell$, the Taylor expansion of $u$ at $t$ leads to the following expression
\begin{equation}
u^{(\ell)}(t+s) = \sum_{j=0}^{r}u^{(\ell+j)}(t)\frac{s^j}{j!} + R_\ell(t,s)\ ,\ \forall \ s\in\RR\ .
\end{equation}
where $R_\ell(t,s) = \int_0^{s}\int_0^{s_1}\dots\int_0^{s_{r}}  u^{(\ell+r+1)}(t+s_{r}) ds_1\dots ds_{r}$.
Hence this implies
\begin{align*}
\bar u^{(r)}(t+s)
= \mat_r(\bar u^{(2r)}(t))
D_r
V_r(s) + \begin{bmatrix}
R_0(t,s)\\
\vdots
\\
R_{r}(t,s)
\end{bmatrix}\ ,
\end{align*}
where $D_r = \texttt{Diag}\left\{1, 1, \dots, \frac{1}{r!}\right\}$ and $V_r(s) = \begin{bmatrix}
1 & s &\dots & s^{r}
\end{bmatrix}^\top$.
Since by assumption, $u$ is differentially exciting of order $r$ at time $t$, this implies the following equality.
$$
\int_{t}^{t+\epsilon} 
\bar{u}^{(r)}(s)
\left (\bar {u}^{(r)}(s)\right )^\top
ds = \mat_r(\bar u^{(2r)}(t))P(\epsilon,t)\mat_r(\bar u^{(2r)}(t))^\top
$$
where
\begin{align*}
P(\epsilon &,t) = D_r
\begin{bmatrix}\!\!
\epsilon & \frac{\epsilon^2}{2} & \frac{\epsilon^3}{3} & \dots & \frac{\epsilon^{r+1}}{r+1}\\
\frac{\epsilon^2}{2} & \frac{\epsilon^3}{3} & \dots & \frac{\epsilon^{r+1}}{r+1} & \frac{\epsilon^{r+2}}{(r+2)}\\ \vdots\\ \frac{\epsilon^{r+1}}{r+1} & & \dots&& \frac{\epsilon^{2r+1}}{2r+1} \!\!\end{bmatrix}D_r\\
&+ D_rN(\epsilon,t)(\mat_r(\bar u^{(2r)}(t))^\top)^{-1} + \mat_r(\bar u^{(2r)}(t))^{-1} N(\epsilon,t)^\top D_r
\end{align*}
where $N(\epsilon,t)$ is the $(r+1)\times (r+1)$ real matrix defined as $N(\epsilon,t)=\int_0^\epsilon V_r(s)\begin{bmatrix}
R_0(s,t) & \dots & R_{r}(s,t)
\end{bmatrix}$.

To show that inequality
 (\ref{eq_StandardDiffExc}) holds, it is needed to show that matrix $P$ is positive definite for sufficiently small $\epsilon$.
In order to show this, let $D_\epsilon = \texttt{Diag}\{1, \epsilon, \dots, \epsilon^{r}\}$.
The matrix $P$ can be decomposed as follows.
\begin{equation}\label{eq_MatP}
P(\epsilon,t)=D_\epsilon\epsilon\left( D_rH_r D_r+Q(\epsilon,t)\right)D_\epsilon\ .
\end{equation}
where $H_r$ is the Hilbert matrix defined as $H_r=\begin{bmatrix}
1 & \frac{1}{2} & \frac{1}{3} & \dots & \frac{1}{r+1}\\
\frac{1}{2} & \frac{1}{3} & \dots & \frac{1}{r} & \frac{1}{(r+2)}\\
\vdots\\
\frac{1}{r} & & \dots&& \frac{1}{2r+1}
\end{bmatrix}$ and
$Q(\epsilon,t)$ is the matrix defined as
\begin{align*}
Q(\epsilon,t) &= \dfrac{D_r D_\epsilon^{-1} N(\epsilon,t)
(\mat_r(\bar u^{(2r)}(t))^\top)^{-1}D_\epsilon^{-1}}{\epsilon}\\ 
&\qquad \qquad \qquad+ \dfrac{D_\epsilon^{-1}\mat_r(\bar u^{(2r)}(t))^{-1}
N(s,t)^\top D_\epsilon^{-1} D_r}{\epsilon}\ .
\end{align*}
The Hilbert matrix being positive definite, it implies that $P$ is positive definite for sufficiently small $\epsilon$ if  the norm of the matrix 
$Q$ goes to zero as $\epsilon$ goes to zero.
In order to upper bound the norm of $Q$, the following inequality can be obtained.
$$
|R_\ell(t,s)|\leq \sup_{\nu\in[0,s]}\left |u^{(\ell+r+1)}(t+\nu)\right |\frac{s^{r+1}}{(r+1)!}\ .
$$
This leads to the following inequality.
\begin{align*}
\left |\left (\frac{D_\epsilon^{-1} N(\epsilon,t)}{\epsilon}\right )_{i,\ell}\right | &= \left |\epsilon^{-i}\int_0^\epsilon s^{i-1} R_\ell(t,s) ds\right |\\
&\leq  \sup_{\nu\in[0,\epsilon]}\left |u^{(\ell+r+1)}(t+\nu)\right |
\frac{\epsilon^{r+1}}{(i+r+1)(r+1)!}
\end{align*}
Hence, for $\epsilon<1$, it yields
\begin{align*}
|Q(\epsilon,t)| &\leq 2 |D_r| \left | \frac{D_\epsilon^{-1} N(\epsilon,t)}{\epsilon}\right | \left |(\mat_r(\bar u^{(2r)}(t))^\top)^{-1}\right | \left |D_\epsilon^{-1} \right |  \\
&\leq 2 (1+r)\sup_{\ell\in[0,r],\nu\in[0,\epsilon]}\left |u^{(\ell+r+1)}(t+\nu)\right | \\
&
\null\qquad\qquad\times\dfrac{\epsilon^{1+r}}{(2+r)(r+1)!} \left |\mat_r(\bar u^{(2r)}(t))^{-1}\right | \epsilon^{-r} 
\end{align*}
This gives finally
$$
|Q(\epsilon,t)| \leq \epsilon 2(r+1)
\frac{\sup_{\ell\in[r+1,2r+1],\nu\in[0,1]}\left |u^{(\ell)}(t+\nu)\right |}{\sqrt{\rho_u(t)}(2+r)(r+1)!}
$$
where $\rho_u(t)$ is a positive real number such that $\mat_r(\bar u^{(2r)}(t))^\top \mat_r(\bar u^{(2r)}(t)) \geq \rho_u(t) I$ which exists since $u$ is differentially exciting of order $r$ at time $t$.

This implies that for $\epsilon$ sufficiently small $|Q(\epsilon,t)|$ becomes small.
This allows to say that the matrix $P$ defined in (\ref{eq_MatP}) is positive definite for small $\epsilon$.
Consequently, inequality 
 (\ref{eq_StandardDiffExc}) holds and the result follows.
}

\subsection{Proof of Proposition \ref{Prop_AssIndentif}}
\label{Sec_ProofPropAssIndetif}
This proof is decomposed into two parts.
In a first part the injectivity of the mapping $\HH_{4n-1}$ is demonstrated.
Then it is shown that it is also full rank. 
From this, the existence of the positive real number $L_\HH$ is obtained employing \cite[Lemma 3.2]{Andrieu_SICON_2014}.

\noindent\textbf{Part 1: Injectivity}
Assume there exist $(x,\theta)$ and $(x^*,\theta^*)$ both in $\CR_x\times\CR_\theta$ and $v=(v_0, \dots, v_{4n-2})$ in $U_{4n}$ such that
$
\HH_{4n-1}(x,\theta,v) = \HH_{4n-1}(x^*,\theta^*,v)
$.
To simplify the notation, let us denote $y_j= \left(\HH_{4n-1}(x,\theta,v)\right)_{j+1}=\left(\HH_{4n-1}(x^*,\theta^*,v)\right)_{j+1}$ for $j=0, \dots, 4n-1$. 
Note that for all $j\geq n$ we have
\begin{equation}\label{Dyn_diffy}
y_j = -\theta_{an}y_{j-n} - \dots - \theta_{1} y_{j-1} + \theta_{bn} v_{j-n} + \dots + \theta_{b1} v_{j-1}\ .
\end{equation}
It follows that the following set of $3n$ equations holds, 
\begin{equation}\label{eq_CondGen}
\begin{array}{ll}
0 &= \begin{bmatrix}
(\HH_r(x,\theta,v) -\HH_r(x^*,\theta^*,v))_{n}\\
(\HH_r(x,\theta,v) -\HH_r(x^*,\theta^*,v))_{n+1}\\
\vdots\\
(\HH_r(x,\theta,v) -\HH_r(x^*,\theta^*,v))_{4n-1}
\end{bmatrix}\\
\\
&=\begin{bmatrix}  
y &  \dots & y_{n-1} & v_0 & \dots & v_{n-1}\\
y_1&  \dots & y_{n} & v_1 & \dots & v_{n}\\
\vdots & \vdots & \vdots \\
y_{3n-1} & \dots & y_{4n-2} & v_{3n-1} & \dots & v_{4n-2}
\end{bmatrix}\Delta

\end{array}
\end{equation}
where $\Delta=\begin{bmatrix}\delta_{an}&\cdots&\delta_{a1}&\delta_{bn}&\cdots&\delta_{b1}\end{bmatrix}^\top$, $\delta_{aj}=\theta_{aj}^*-\theta_{aj}$ and $\delta_{bj}=\theta_{bj}-\theta_{bj}^*$.
This yields for $\ell=0,\dots,2n-1 $
$$
\begin{bmatrix}  
y_{\ell} &  \dots & y_{\ell+n-1} & v_\ell & \dots & v_{\ell+n-1}\\
 y_{\ell+1} &  \dots & y_{\ell+n} &  v_{\ell+1} & \dots & v_{\ell+n}\\
 \vdots & \vdots & \vdots \\
y_{\ell+n} &  \dots & y_{\ell+2n-1} & v_{\ell+n} & \dots & v_{\ell+2n-1}
\end{bmatrix}
\Delta=0\ .
$$
Hence, employing equality (\ref{Dyn_diffy}) on the last line of the previous vector and multiplying the previous vector by
$
\begin{bmatrix}
\theta_{an} & \theta_{a(n-1)} & \dots & \theta_{a1} &1
\end{bmatrix}
$
leads to an algebraic equation depending only on $v$ in the form 
\begin{equation}
\label{eq_u}\sum_{j=0}^{2n-1} c_j v_{\ell+j} = 0\ ,\ \ell=0,\dots, 2n-1\ .
\end{equation}
where
$$
\begin{aligned}
&c_0 = \delta_{bn}\theta_{an} + \delta_{an}\theta_{bn} \\
&c_1 = \theta_{an}\delta_{b(n-1)} + \theta_{a(n-1)}\delta_{bn} + \theta_{b(n-1)}\delta_{an} + \theta_{bn}\delta_{a(n-1)}
\end{aligned}
$$
and more generally, the $c_j$ are given by the matrix definition
$$\begin{bmatrix}c_0&\cdots&c_{2n-1}\end{bmatrix}=\mathcal M(\thetaa,\thetab)\Delta$$
where $\mathcal M(\thetaa,\thetab)$ is the Sylvester real matrix.\longue{defined as

\begin{equation}\label{eq_SylvMatrix}
\mathcal M(\thetaa,\thetab) = \begin{bmatrix}
\theta_{bn}     & \cdots& 0             &\theta_{an}    & \cdots& 0         \\
\theta_{b(n-1)} & \ddots&\vdots         &\vdots         & \ddots& \vdots    \\
\vdots          & \ddots&\theta_{bn}    &\theta_{a1}    & \ddots&\theta_{an}\\
\theta_{b1}     & \ddots&\theta_{b(n-1)}& 1             & \ddots&\vdots     \\ 
\vdots          & \ddots&\vdots         & \vdots        & \ddots&\theta_{a1}\\
 0              &\cdots &\theta_{b1}    & 0             & \cdots&1
\end{bmatrix} \ .
\end{equation}
}
Finally, this can be rewritten
$$
\mat_{n-1}(v) \mathcal M(\thetaa,\thetab) \Delta =0
$$
Note that due to the particular structure of the  couple $(A(\thetaa),B(\thetab))$ and with controllability property, it implies that the Sylvester matrix 
is invertible for all $(\theta_a,\theta_b)$ in $\CL(\CR_\theta)$.
The matrix  $\mat_{n-1}(v)$ being also invertible by assumption, this implies that $0=\delta_{a1}=\dots = \delta_{an}= \delta_{b1}=\dots = \delta_{bn}$ and consequently $\theta = \theta^*$.
From the observability property of the couple $(A(\theta_a),C)$, this yields that $x=x^*$.
We conclude injectivity of the mapping $\HH_{4n-1}$ with respect to $(x,\theta)$.\\

\textbf{Part 2: The mapping $\HH_{4n-1}$ is full rank}.\\
Let $\HH_{4n-1}$ satisfy the following equation 
\begin{equation}\label{RankCond}
\frac{\partial \HH_{4n-1}}{\partial x}(x,\theta,v)v_x + \frac{\partial \HH_{4n-1}}{\partial \theta}(x,\theta,v)v_\theta = 0.
\end{equation}\label{RankForm}
Then, we must prove that $[v_x^T ~~~~ v_\theta^T]^T=0$ for all $(x,\theta)$ in $\CR_x \times \CR_\theta$. We have for $i=0,...,4n-1$
$$
\frac{\partial y_i}{\partial (x,\theta)}\begin{bmatrix}
v_x\\v_\theta
\end{bmatrix} = 0
$$
and for $i=n,...,4n-1$, 
$$
y_i = \theta_{an}y_{i-n} + \cdots +\theta_{a1}y_{i-1} + \theta_{bn}v_{i-n} + \cdots +\theta_{b1}v_{i-1}
$$
So
\begin{align*}
\frac{\partial y_i}{\partial (x,\theta)}\begin{bmatrix}
v_x\\v_\theta
\end{bmatrix} &= \theta_{an}\frac{\partial y_{i-n}}{\partial (x,\theta)}\begin{bmatrix} v_x\\v_\theta \end{bmatrix}+\cdots+ \theta_{a1}\frac{\partial y_{i-1}}{\partial (x,\theta)}\begin{bmatrix} v_x\\v_\theta \end{bmatrix}\\
&+ \begin{bmatrix} y_{i-1} & \cdots & y_{i-n} & v_{i-1} &\cdots& v_{i-n}\end{bmatrix}v_\theta\\
&=\begin{bmatrix} y_{i-1} & \cdots &y_{i-n} & v_{i-1} &\cdots& v_{i-n} \end{bmatrix}v_\theta
\end{align*}
And so, if we follow exactly the same reasoning as in the previous step and consider the same assumptions, one can conclude that $v_\theta=0$.

 On the other hand, we can write $\HH_{4n-1}(x,\theta,v)$ in the following form
$$
\HH_{4n-1}(x,\theta,v) = H_{4n-1}(\theta)x + \sum_{j=1}^{4n-2}S^j H_{4n-1}(\theta)B(\theta)v_{j-1}\ ,
 $$
and from (\ref{RankCond}) we get 
$$\frac{\partial \HH_{4n-1}}{\partial x}v_x+\frac{\partial \HH_{4n-1}}{\partial \theta}v_\theta=H_{4n-1}(\theta)v_x+\frac{\partial \HH_{4n-1}}{\partial \theta}v_\theta=0.$$
But we have proved that $v_\theta=0$. Therefore 
$H_{4n-1}(\theta)v_x=0$. Since we assume that $y$ is observable for each $\theta$ in $\CR_\theta$, $H_{4n-1}(\theta)$ is full column rank  so that $v_x=0$ and the result follows.

\subsection{Proof of Proposition \ref{Prop_ExplTinv}}
\label{Sec_ProofPropExplTinv}
The proof of this result is made in three steps.
In a first step it is shown that the function $T$ is solution to an implicit equation in which the unknown $x$ and $\theta$ appear linearly.
In a second step, it is shown that the linear matrix which appears is full rank. Finally, the selection of $p_{\min}$ is made.

\noindent \textbf{Step 1} Proof that (\ref{def_Pi}) holds
Indeed,  note that the function $M_i(\theta)$ can be rewritten
\begin{equation}\label{eq:Inv1}
	M_i(\theta) = C  (A(\thetaa)-\lambda_i I_n)^{-1}\ .
\end{equation}
Let
$$J_i=\begin{bmatrix} -\lambda_i&1&\cdots&0\\ \vdots&\ddots&\ddots&\vdots\\ 0&\cdots&-\lambda_{i}&1\\0 &\cdots&0&-\lambda_i\end{bmatrix} \in \Re^{n\times n} 
\ ,$$
then
$A(\thetaa)-\lambda_i I_n = J_i -\theta_a C$.
Applying the Sherman-Morrison-Woodbury formula, one gets:
\begin{equation}\label{eq:Inv2}
(A(\thetaa)-\lambda_i I_n)^{-1}=J_i^{-1}+\frac{J_i^{-1}\thetaa C J_i^{-1}}{1-CJ_i^{-1}\thetaa}
\end{equation}
where  $1-e_1^TJ_i^{-1}\thetaa\neq 0$ is obtained from \eqref{eq_DisjointEV}.
Combining \eqref{eq:Inv1} and \eqref{eq:Inv2} gives us
$
(1-C J^{-1}_i \thetaa)M_i(\thetaa)=(CJ_i^{-1})
$,
which, together with \eqref{eq_Ti}, reveals that: 
$$
\left (1-CJ^{-1}_i \thetaa\right ) T_i(x,w) = C J^{-1}_ix - CJ^{-1}_i B(\theta_b) w_i\ . 
$$
it can be verified that $V_i^T=e_1^TJ_i^{-1}$. 
Rearranging the previous expression, one gets the proof of \eqref{def_Pi}. 

\noindent \textbf{Step 2:} 
The matrix $P$ defined in \eqref{def_Pi} has full column rank for all $(z,x,\theta,w)$ such that $z=T(x,\theta,w)$ with $(x,\theta,w)$ in $\CL(\CR_x)\times\CL(\CR_\theta)\times\CL(\CR_w)$.

Differentiating \eqref{def_Pi} with respect to $(x,\theta)$ yields for all $(x,\theta,w)$
$$
\frac{\partial T_i}{\partial (x,\theta)}(x,\theta,w) = \frac{\partial T_i}{\partial (x,\theta)}(x,\theta,w)  V_i^\top\thetaa + P_i(T_i(x,\theta,w) ,w)
$$
This implies that
$
[1-V_i^\top\thetaa]\frac{\partial T_i}{\partial (x,\theta)}(x,\theta,w) =  P_i(T_i(x,\theta,w) ,w)
$.
Hence:
$$
\diag \left \{1-V_1^T \theta_a,...,1-V_r^T \theta_a\right \}\frac{\partial T}{\partial (x,\theta)}(x,\theta,w) = P(z,w)
$$
Again, since condition \eqref{eq_DisjointEV}  holds, it yields that  $\diag \left \{1-V_1^T \theta_a,...,1-V_r^T \theta_a\right \}$ is invertible.
Moreover, since (\ref{eq_InjTcontrGen}) is satisfied, it implies that $\frac{\partial T}{\partial (x,\theta)}(x,\theta,w)$ is full column rank.
Consequently $P$ is full column rank for all $(z,x,\theta,w)$ such that $z=T(x,\theta,w)$ with $(x,\theta,w)$ in  $\CL(\CR_x)\times\CL(\CR_\theta)\times\CL(\CR_w)$.

\noindent\textbf{Step 3:} Conclusion:
Finally, let for all $(x,\theta,w)\in\CL(\CR_x)\times\CL(\CR_\theta)\times\CL(\CR_w)$ 
\begin{align}
p_{\min} = \dfrac{1}{2} \min \sigma_{\min}\left \{P(T(x,\theta,w),w)^\top P(T(x,\theta,w),w)\right \}.
\end{align}
With this definition, it yields that given $(x,w,\theta)$ in $ \CL(\CR_x)\times\CL(\CR_\theta)\times\CL(\CR_w)$, we have
$$
P(T(x,\theta,w),w)^\top P(T(x,\theta,w),w)\geq p_{\min}I_{3n}\ .
$$
Hence, with the mapping $T^*$ defined in (\ref{ExpTInv}) it yields, 
\begin{align*}
T^*(T(x,\theta,w)) &= ( P(T(x,\theta,w),w)^\top P(T(x,\theta,w),w))^{-1}\\
&\hspace{6em}\times P(T(x,\theta,w),w)^\top T(x,\theta,w) 
\end{align*}
which gives from \eqref{def_Pi}
$$
T^*(T(x,\theta,w),w)  = \begin{bmatrix}
 x^\top & \theta_a^\top & \theta_b^\top\end{bmatrix}^\top \ .
$$
Let also $\CR_{zw}$ be the open subset of $\RR^{2r}$ such that 
$$
\CR_{zw} = \!\Big\{(z,w), z=T(x,\theta,w), \sigma_{\min} \!\big \{\!P(z,w)^\top \! P(z,w)\big\}> p_{\min}\!\Big\}\ .
$$
Note that $\CR_{z,w}$ is an open subset in which $T^*$ is smooth and which contains the compact  set $\{(z,w), z=T(x,\theta,w), x\in\CR_x, \theta\in\CR_\theta\}$. 
Let $\epsilon_{T}$ be a positive real number sufficiently small such that the compact set
\begin{align*}
\CL(\CR_{z}) =& \{z\in \RR^r, \exists (x,\theta,w)\in \CL(\CR_x)\times\CL(\CR_\theta)\times\CL(\CR_w) ,\\
& |z-T(x,\theta,w)|\leq \epsilon_T\}
\end{align*}
satisfies $\CL(\CR_{z}) \times \CL(\CR_{w})\subset  \CR_{zw} $.
Note that $T^*$ is Lipschitz in $\CL(\CR_{z}) \times \CL(\CR_{w})$. Hence  the result holds for a particular $L_{T^*}$.

\bibliography{BibVA}

\begin{thebibliography}{10}

\bibitem{AfriAndrieuBakoDufour_ACC_15}
C.~Afri, V.~Andrieu, L.~Bako, and P.~Dufour.
\newblock Identification of linear systems with nonlinear luenberger observers.
\newblock In {\em American Control Conference (ACC), 2015}, pages 3373--3378,
  July 2015.

\bibitem{Afrietal_TAC_2017}
C.~Afri, V.~Andrieu, L.~Bako, and P.~Dufour.
\newblock {State and parameter estimation: a nonlinear Luenberger observer
  approach}.
\newblock {\em Automatic Control, IEEE Transactions on}, 2017.

\bibitem{Andrieu_SICON_2014}
V.~Andrieu.
\newblock Convergence speed of nonlinear luenberger observers.
\newblock {\em SIAM Journal on Control and Optimization}, 52(5):2831--2856,
  2014.

\bibitem{AndrieuEytardPraly_CDC_14_DynExtInvObs}
V.~Andrieu, J.-B. Eytard, and L.~Praly.
\newblock Dynamic extension without inversion for observers.
\newblock In {\em Decision and Control (CDC), 2014 IEEE 53rd Annual Conference
  on}, pages 878--883, Dec 2014.

\bibitem{AndrieuPraly_CDC_2004}
V.~Andrieu and L.~Praly.
\newblock Remarks on the existence of a kazantzis-kravaris/luenberger observer.
\newblock In {\em Decision and Control, 2004. CDC. 43rd IEEE Conference on},
  volume~4, pages 3874--3879 Vol.4, Dec 2004.

\bibitem{BastinGevers_TAC_88}
G.~Bastin and M.R. Gevers.
\newblock Stable adaptive observers for nonlinear time-varying systems.
\newblock {\em IEEE Transactions On Automatic Control}, 33(7):650--658, 1988.

\bibitem{BernardAndrieuPraly_15_HAL_NonLinObsOriCoord}
P.~Bernard, V.~Andrieu, and L.~Praly.
\newblock {Nonlinear observer in the original coordinates with diffeomorphism
  extension and jacobian completion}.
\newblock Submitted for publication in SIAM Journal of Control and
  Optimization, September 2015.

\bibitem{Goodwin_Book_14_Adaptive}
G.~C. Goodwin and K.~S. Sin.
\newblock {\em Adaptive filtering prediction and control}.
\newblock Courier Corporation, 2014.

\bibitem{IoannouSun_Book_2012robust}
P.~A. Ioannou and J.~Sun.
\newblock {\em Robust adaptive control}.
\newblock Courier Dover Publications, 2012.

\bibitem{Janecki_87_SCL_persistency}
D.~Janecki.
\newblock Persistency of excitation for continuous-time systems—time-domain
  approach.
\newblock {\em Systems \& Control Letters}, 8(4):333--344, 1987.

\bibitem{KazantzisKravaris_SCL_98}
N.~Kazantzis and C.~Kravaris.
\newblock {Nonlinear observer design using Lyapunov's auxiliary theorem}.
\newblock {\em Systems \& Control Letters}, 34:241--247, 1998.

\bibitem{Kreisselmeier_TAC_77}
G.~Kreisselmeier.
\newblock Adaptive observers with exponential rate of convergence.
\newblock {\em IEEE Transactions on Automatic Control}, AC-22(1):2--8, 1977.

\bibitem{KreisselmeierEngel_TAC_03}
G.~Kreisselmeier and R.~Engel.
\newblock {Nonlinear observers for autonomous Lipschitz continuous systems}.
\newblock {\em IEEE Transactions on Automatic Control}, 48(3), 2003.

\bibitem{Luenberger_IEEE_TME_64}
D.~Luenberger.
\newblock {Observing the state of a linear system}.
\newblock {\em IEEE Transactions on Military Electronics}, MIL-8:74--80, 1964.

\bibitem{MarconiPraly_TAC_08}
L.~Marconi and L.~Praly.
\newblock {Uniform practical nonlinear output regulation}.
\newblock {\em IEEE Transactions on Automatic Control}, 53(5):1184--1202, 2008.

\bibitem{MarinoTomei_TAC_95}
R.~Marino and P.~Tomei.
\newblock Adaptive observers with arbitrary exponential rate of convergence for
  nonlinear systems.
\newblock {\em IEEE Transactions On Automatic Control}, 40(7):1300--1304, 1995.

\bibitem{MarinoTomei_Book_95}
R.~Marino and P.~Tomei.
\newblock {\em {Nonlinear control design: geometric, adaptive and robust}}.
\newblock Prentice Hall International (UK) Ltd., 1995.

\bibitem{Mcshane_BAMS_34}
E.J. McShane.
\newblock {Extension of range of functions}.
\newblock {\em Bulletin of the American Mathematical Society}, 40(12):837--842,
  1934.

\bibitem{NarendraAnnaswamy_Book_89}
K.~S. Narendra and A.~M. Annaswamy.
\newblock {\em Stable Adaptive Systems}.
\newblock Prentice Hall, 1989.

\bibitem{PralyEtAl_MTNS_06_ObsOscillator}
L.~Praly, A.~Isidori, and L.~Marconi.
\newblock A new observer for an unknown harmonic oscillator.
\newblock In {\em 17th International Symposium on Mathematical Theory of
  Networks and Systems}, pages 24--28, 2006.

\bibitem{RapaportMaloum_IJRNC_04}
A.~Rapaport and A.~Maloum.
\newblock {Design of exponential observers for nonlinear systems by embedding}.
\newblock {\em International Journal of Robust and Nonlinear Control},
  14(3):273--288, 2004.

\bibitem{ShimkinFeuer_87_SCLpersistency}
N.~Shimkin and A.~Feuer.
\newblock Persistency of excitation in continuous-time systems.
\newblock {\em Systems \& control letters}, 9(3):225--233, 1987.

\bibitem{Shoshitaishvili_TSA_90}
A.N. Shoshitaishvili.
\newblock {Singularities for projections of integral manifolds with
  applications to control and observation problems}.
\newblock {\em Theory of singularities and its applications}, 1:295, 1990.

\bibitem{Q_Zhang_2002}
Q.~Zhang.
\newblock Adaptive observer for multiple-input multiple-output ({MIMO}) linear
  time-varying systems.
\newblock {\em IEEE Transactions On Automatic Control}, 47(3):525--529, 2002.

\end{thebibliography}

\end{document}